\documentclass[12pt]{scrartcl}
\usepackage[english]{babel}
\usepackage[utf8x]{inputenc}
\usepackage{amsfonts}
\usepackage{amsmath}
\usepackage{amssymb}
\usepackage{amsthm}
\usepackage{mathrsfs}
\usepackage[colorlinks=false,urlcolor=blue]{hyperref}
\usepackage{breakurl}
\usepackage{listings}
\usepackage{graphicx}
\usepackage{enumerate}
\usepackage{tikz}
\usetikzlibrary{patterns}

\addtokomafont{section}{\rmfamily}
\addtokomafont{subsection}{\rmfamily}
\addtokomafont{paragraph}{\rmfamily}
\usepackage{indentfirst}

\swapnumbers
\theoremstyle{plain}
\newtheorem{satz}{Theorem}[section]
\newtheorem{lem}[satz]{Lemma}
\newtheorem{kor}[satz]{Corollary}
\newtheorem{prop}[satz]{Proposition}
\theoremstyle{definition}
\newtheorem{defn}[satz]{Definition}
\newtheorem{bem}[satz]{Remark}

\newcommand{\R}{\mathbb{R}}
\newcommand{\C}{\mathbb{C}}
\newcommand{\N}{\mathbb{N}}

\renewcommand{\O}{\mathbb{O}}

\newcommand{\Ric}{\operatorname{Ric}}
\newcommand{\scal}{\operatorname{scal}}
\newcommand{\vol}{\operatorname{vol}}
\newcommand{\Sym}{\operatorname{Sym}}
\newcommand{\tr}{\operatorname{tr}}
\newcommand{\im}{\operatorname{im}}

\newcommand{\diag}{\operatorname{diag}}
\newcommand{\sym}{\operatorname{sym}}

\newcommand{\End}{\operatorname{End}}
\newcommand{\Aut}{\operatorname{Aut}}
\newcommand{\Hom}{\operatorname{Hom}}
\newcommand{\Ad}{\operatorname{Ad}}
\newcommand{\ad}{\operatorname{ad}}
\newcommand{\Cas}{\operatorname{Cas}}

\newcommand{\Sy}{\mathscr{S}}
\newcommand{\Sl}{\mathfrak{S}}
\newcommand{\X}{\mathfrak{X}}
\newcommand{\U}{\operatorname{U}}
\newcommand{\SU}{\operatorname{SU}}

\newcommand{\SO}{\operatorname{SO}}
\newcommand{\so}{\mathfrak{so}}
\newcommand{\Sp}{\operatorname{Sp}}
\newcommand{\Spin}{\operatorname{Spin}}

\newcommand{\pr}{\operatorname{pr}}
\renewcommand{\i}{\mathrm{i}}

\renewcommand{\H}{\mathfrak{H}}
\newcommand{\Einstein}{E}
\newcommand{\TT}{\Sy^2_{\mathrm{tt}}}
\renewcommand{\t}{\mathfrak{t}}

\title{\rmfamily Stability of Einstein metrics on symmetric spaces of compact type}
\author{Paul Schwahn*}
\date{\today}

\begin{document}

\maketitle
{\let\thefootnote\relax\footnotetext{*Institut für Geometrie und Topologie, Fachbereich Mathematik, Universität Stuttgart, Pfaffenwaldring 57, 70569 Stuttgart, Germany.}}

\begin{abstract}
\footnotesize
\begin{center}
\textbf{Abstract}
\end{center}

\noindent
We prove the linear stability with respect to the Einstein-Hilbert action of the symmetric spaces $\SU(n)$, $n\geq 3$, and $E_6/F_4$. Combined with earlier results, this resolves the stability problem for irreducible symmetric spaces of compact type.

\textit{MSC (2020):} 53C24, 53C25, 53C30, 53C35.

\textit{Keywords:} Symmetric spaces, Einstein metrics, Stability, Lichnerowicz Laplacian.
\end{abstract}

\subsection*{Declarations}

\paragraph{Funding} No funds, grants, or other support was received.

\paragraph{Employment} During the time of research as well as currently, the author is employed as a research assistant at the University of Stuttgart.

\paragraph{Competing interests} The author has no relevant financial or non-financial interests to disclose.

\paragraph{Data availability statement} Data sharing not applicable to this article as no datasets were generated or analysed during the current study.

\paragraph{Code availability} Not applicable.

\pagebreak

\section{Introduction}

Let $M$ be a closed manifold of dimension $n>2$. It is a well-known fact (see \cite{Be}) that Einstein metrics are critical points of the total scalar curvature functional
\[g\mapsto S(g)=\int_M\scal_g\vol_g,\]
also called the Einstein-Hilbert action, restricted to the space of Riemannian metrics of a fixed volume. In general, these critical points are neither maximal nor minimal. If we, however, restrict $S$ to the set $\Sl$ of all Riemannian metrics on $M$ of the same fixed volume that have constant scalar curvature, then some Einstein metrics are maximal, while others form saddle points. To examine this, one considers the second variation $S''_g$ of $S$ at a fixed Einstein metric $g$ on $M$. If we exclude the case where $(M,g)$ is a standard sphere, the tangent space of $\Sl$ at $g$ consists precisely of tt-tensors, i.e. symmetric $2$-tensors that are \textbf{t}ransverse (divergence-free) and \textbf{t}raceless. In these directions, the coindex and nullity of $S''_g$ are always finite. The \emph{stability problem} is to decide whether they vanish for a given Einstein manifold $(M,g)$.

The stability of an Einstein metric $g$ is determined by the spectrum of a Laplace-type operator $\Delta_L$, called the Lichnerowicz Laplacian, on tt-tensors. There is a critical eigenvalue, corresponding to null directions for $S''_g$, which is equal to $2\Einstein$, where $\Einstein$ is the Einstein constant of $g$. The metric $g$ is called \emph{linearly (strictly) stable} if $\Delta_L\geq2\Einstein$ (resp. $\Delta_L>2\Einstein$) on tt-tensors, and \emph{infinitesimally deformable} if there is a tt-eigentensor of $\Delta_L$ for the critical eigenvalue.

Suppose that $(M,g)$ is a locally symmetric Einstein manifold of compact type. The Cartan–Ambrose–Hicks theorem implies that its universal cover $(\tilde{M},\tilde{g})$ is a simply connected symmetric space. As such, $(\tilde{M},\tilde{g})$ can be written as a Riemannian product of irreducible symmetric spaces of compact type. For many of these spaces, the stability problem has been decided by N. Koiso. The following theorem collects the results of Koiso in \cite{Kois} together with a result of J. Gasqui and H. Goldschmidt in \cite{GG} about the complex quadric $\SO(5)/(\SO(3)\times\SO(2))$.

\begin{satz}
\label{koisoclass}
\begin{enumerate}
	\item The only irreducible symmetric spaces of compact type that are infinitesimally deformable are
\[\SU(n),\ \SU(n)/\SO(n),\ \SU(2n)/\Sp(n)\quad(n\geq3),\]
\[\SU(p+q)/\mathrm{S}(\U(p)\times\U(q))\quad (p\geq q\geq 2),\]
as well as $E_6/F_4$.
	\item The irreducible symmetric spaces
\[\Sp(n)\quad(n\geq 2),\quad\Sp(n)/\U(n)\quad(n\geq 3),\]
as well as the complex quadric $\SO(5)/(\SO(3)\times\SO(2))$ are unstable.
	\item Let $(M,g)$ be an irreducible symmetric space of compact type. If $(M,g)$ is none of the spaces from 1. and 2., nor one of
\[\Sp(p+q)/(\Sp(p)\times\Sp(q))\quad(p\geq q\geq2\text{ or }p=2,q=1)\]
nor $F_4/\Spin(9)$, then $g$ is strictly stable.
\end{enumerate}
\end{satz}

Moreover, the smallest eigenvalue of $\Delta_L$ on trace-free symmetric $2$-tensors has been computed in each case (see \cite{CH}). Among the spaces that possess infinitesimal deformations, we have $\Delta_L\geq 2\Einstein$ on $\Sy^2_0(M)$ on the spaces
\[\SU(n)/\SO(n),\ \SU(2n)/\Sp(n)\ (n\geq 3),\ \SU(p+q)/\text{S}(\U(p)\times\U(q))\ (p\geq q\geq 2),\]
which shows that they are linearly stable.

However, this did not fully settle the stability problem on irreducible symmetric spaces of compact type. In particular, it had not been decided whether unstable directions exist on the spaces
\[\SU(n)\quad (\text{where }n\geq 3),\qquad E_6/F_4,\qquad F_4/\Spin(9),\]
\[\Sp(p+q)/(\Sp(p)\times\Sp(q))\quad (\text{where }p\geq q\geq2\text{ or }p=2,q=1).\]
In these cases, we know that $\Delta_L$ has eigenvalues smaller than $2\Einstein$ on the space of trace-free symmetric $2$-tensors, but it had not been checked whether the corresponding eigentensors are also divergence-free. In a recent paper \cite{SW}, U. Semmelmann and G. Weingart show the following results.

\begin{satz}
\label{sw}
\begin{enumerate}
	\item The quaternionic Grassmannians $\Sp(p+q)/(\Sp(p)\times\Sp(q))$ are linearly stable for $p=2$ and $q=1$, but unstable for $p\geq q\geq 2$.
	\item The Cayley plane $\O P^2=F_4/\Spin(9)$ is linearly stable.
\end{enumerate}
\end{satz}

The current article finally resolves the question of stability for the last remaining cases by proving the following.

\begin{satz}
\label{result}
The symmetric spaces $\SU(n)$, where $n\geq 3$, as well as $E_6/F_4$ are linearly stable.
\end{satz}

Consider a manifold $(M,g)$ that is a Riemannian product of Einstein manifolds. Then $(M,g)$ is Einstein if and only if the factors have the same Einstein constant $\Einstein$. It turns out that if $\Einstein>0$, then $(M,g)$ is always unstable (see \cite{Kr}, Prop. 3.3.7). For example, if $(M,g)$ is the Riemannian product of two Einstein manifolds $(M_i^{n_i},g_i)$ ($i=1,2$) with the same Einstein constant, then an unstable direction is given by
\[h:=n_2\pi_1^\ast g_1-n_1\pi_2^\ast g_2,\]
where $\pi_i: M\to M_i$ are the projections onto each factor, respectively. In particular, a product of symmetric spaces of compact type is always unstable since the factors have positive curvature.

If we take $(M,g)$ to be locally symmetric of compact type, we cannot in general conclude its instability from the instability of its universal cover $(\tilde{M},\tilde{g})$. The same holds for the existence of infinitesimal Einstein deformations. On the other hand, if $(\tilde{M},\tilde{g})$ is infinitesimally non-deformable (resp. stable), then the same follows for $(M,g)$. In \cite{Koi2}, N. Koiso has proved the infinitesimal non-deformability of a large class of such manifolds:

\pagebreak

\begin{satz}
Let $(M,g)$ be a locally symmetric Einstein manifold of compact type. Let $(\tilde{M},\tilde{g})$ be its universal cover and $(\tilde{M},\tilde{g})=\prod_{i=1}^N(M_i,g_i)$ its decomposition into irreducible symmetric spaces.
\begin{enumerate}
	\item For $N=1$, see Theorem \ref{koisoclass}, 1.
	\item If $N=2$ and $M_i$ are neither of the spaces listed in Theorem \ref{koisoclass}, 1., nor $G_2$ or any Hermitian space except $S^2$, then $(M,g)$ is infinitesimally non-deformable.
	\item If $N\geq3$ and $M_i$ are neither of the above nor $S^2$, then $(M,g)$ is infinitesimally non-deformable.
\end{enumerate}
\end{satz}

A closely related notion of stability arises in the study of the Ricci flow. The fixed points (modulo diffeomorphisms and scaling) of the Ricci flow are called Ricci solitons. The \emph{$\nu$-entropy} defined by G. Perelman is a quantity that increases monotonically under the Ricci flow. Its critical points are the \emph{shrinking gradient Ricci solitons}, which include Einstein manifolds. An Einstein metric is called \emph{$\nu$-linearly stable} if the second variation of the $\nu$-entropy is negative-semidefinite. H.-D. Cao, R. Hamilton and T. Ilmanen first studied the $\nu$-linear stability of Einstein metrics (see \cite{CHI}). It turns out that an Einstein metric is $\nu$-linearly stable if and only if $\Delta_L\geq2\Einstein$ on tt-tensors and if the first nonzero eigenvalue of the ordinary Laplacian on functions is bounded below by $2\Einstein$ as well. In particular, $\nu$-linear stability implies linear stability with respect to the Einstein-Hilbert action. In \cite{CH}, the $\nu$-linear stability of irreducible symmetric spaces of compact type is completely decided.

There is yet another notion of stability worth mentioning. It is motivated, for example, by the investigation of Anti-de Sitter product spacetimes and generalized Schwarzschild-Tangherlini spacetimes (see \cite{Di} or \cite{GHP}). An Einstein manifold $(M^n,g)$ with Einstein constant $\Einstein$ is called \emph{physically stable} if
\[\Delta_L\geq\frac{\Einstein}{n-1}\left(4-\frac{1}{4}(n-5)^2\right)=\frac{9-n}{4}\Einstein\]
on tt-tensors. This critical eigenvalue is significantly smaller than the one from stability with respect to the Einstein-Hilbert action, and even negative for $n>9$. As it turns out, all irreducible symmetric spaces of compact type are physically stable (see \cite{Di}). If $(M,g)$ is a product of at least two symmetric spaces of compact type, then the smallest eigenvalue of $\Delta_L$ on tt-tensors is actually equal to $0$; hence $(M,g)$ is physically stable if and only if $n\geq 9$.

In Section \ref{sec:prelim}, we fix the notation and definitions used throughout this work. In particular, we elaborate on the notion of stability of an Einstein metric. In Section \ref{sec:diff}, we recall some tools from the harmonic analysis of homogeneous spaces that are routinely employed. Furthermore, we prove a technical lemma that allows us to make explicit computations involving the divergence operator. A helpful formula for the dimension of tt-eigenspaces of the Lichnerowicz Laplacian is worked out in Section \ref{sec:killing}, generalizing a proposition of Koiso and utilizing properties of Killing vector fields on Einstein manifolds. Section \ref{sec:sun} uses representation theory to determine the stability of $\SU(n)$, making use of the formula from Section \ref{sec:killing}; in Section \ref{sec:ef}, the same is done for $E_6/F_4$. A different approach for proving the stability of both spaces that involves explicit computations of the divergence operator can be found in the Appendix.

\section{Preliminaries}
\label{sec:prelim}

Throughout what follows, let $(M,g)$ be a compact, orientable Riemannian manifold. Let $\nabla$ denote the Levi-Civita connection of $g$. The Riemannian curvature tensor, Ricci tensor and scalar curvature are in our convention given as
\begin{align*}
R(X,Y)Z&:=\nabla_X\nabla_YZ-\nabla_Y\nabla_XZ-\nabla_{[X,Y]}Z,\\
\Ric(X,Y)&:=\tr(Z\mapsto R(Z,X)Y),\\
\scal&:=\tr_g\Ric,
\end{align*}
respectively.\footnote{We use the index $g$ only when the metric-dependence of an object is to be emphasized.} The action of the Riemannian curvature extends to an endomorphism on tensor bundles as
\[R(X,Y)=\nabla_X\nabla_Y-\nabla_Y\nabla_X-\nabla_{[X,Y]},\]
where $\nabla$ also denotes the induced connection on the respective tensor bundle. Furthermore, let $\Sy^p(M)=\Gamma(\Sym^pT^\ast M)$ for $p\geq0$. We denote by
\[\delta:~\Sy^{p+1}(M)\to\Sy^p(M)\]
the \emph{divergence} operator on symmetric tensors, given by
\[\delta=-\sum_ie_i\lrcorner\nabla_{e_i}.\]
The space of tt-tensors, i.e. trace- and divergence-free symmetric $2$-tensors on $M$, is denoted by $\TT(M)$.

Let $\delta^\ast: \Sy^p(M)\to\Sy^{p+1}(M)$ be the formal adjoint\footnote{That is, with respect to the inner product $\langle\cdot,\cdot\rangle_g$ on $\Sym^pT^\ast M$ with orthonormal basis $(e_{i_1}^\flat\odot\ldots\odot e_{i_p}^\flat).$} of the divergence operator. It can be written as
\[\delta^\ast=\sum_ie_i^\flat\odot\nabla_{e_i},\]
where $(e_i)$ is a local orthonormal basis of $TM$. Here, $\odot$ denotes the (associative) symmetric product, defined by
\[\alpha\odot\beta:=\frac{(k+l)!}{k!l!}\sym(\alpha\otimes\beta)\]
for $\alpha\in\Sym^kT$, $\beta\in\Sym^lT$, where $T$ is any vector space and the symmetrization map $\sym: T^{\otimes k}\to\Sym^kT$ is given by
\[\sym(X_1\otimes\ldots\otimes X_k):=\frac{1}{k!}\sum_{\sigma\in S_k}X_{\sigma(1)}\otimes\ldots\otimes X_{\sigma(k)}\]
for $X_1,\ldots,X_k\in T$. This is analogous to the definition of the wedge product via the alternation map. For tensors $\alpha,\beta$ of rank $1$, we have
\[\alpha\odot\beta=\alpha\otimes\beta+\beta\otimes\alpha.\]
It should be noted that $\delta^\ast X^\flat=L_Xg$ for any vector field $X\in\X(M)$. Consequently, the kernel of $\delta^\ast$ on $\Omega^1(M)$ is (via the metric) isomorphic to the space of Killing vector fields on $(M,g)$. More generally, symmetric tensors $\alpha\in\Sy^k(M)$ with $\delta^\ast\alpha=0$ are called \emph{Killing tensors} of rank $k$, and $\delta^\ast$ is sometimes called the \emph{Killing operator}.

\begin{defn}
On tensors of any rank, the following operators are defined:
\begin{enumerate}
	\item The \emph{curvature endomorphism} $q(R)$ is defined by
\[q(R):=\sum_{i<j}(e_i\wedge e_j)_\ast R(e_i,e_j),\]
where $(e_i)$ is a local orthonormal basis of $TM$ and the asterisk indicates the natural action of $\Lambda^2T\cong\so(T)$.
	\item The \emph{Lichnerowicz Laplacian} $\Delta_L$ is defined by
	\[\Delta_L:=\nabla^\ast\nabla+q(R).\]
	Recall that on $\Omega^p(M)$, $p\geq0$, this coincides with the Hodge Laplacian $\Delta$.
\end{enumerate}
\end{defn}

On the space of Riemannian metrics on $M$, which is an open cone in $\Sy^2(M)$, the \emph{total scalar curvature functional} or \emph{Einstein-Hilbert action} is given by
\[S(g)=\int_M\scal_g\vol_g\]
for any Riemannian metric $g$ on $M$. As mentioned earlier, if we restrict this functional to the space of metrics of a fixed total volume, then Einstein metrics are precisely the critical points of the restriction of $S$.

Let $(M,g)$ be an Einstein manifold with Einstein constant $\Einstein\in\R$, that is
\[\Ric=\Einstein g,\]
and suppose that $(M,g)$ is not isometric to a standard round sphere. Denote
\[C^\infty_g(M)=\left\{f\in C^\infty(M)\,\middle|\,\int_Mf\vol_g=0\right\}.\]
It is well known (see \cite{Be}) that there is a decomposition of $\Sy^2(M)$, which is orthogonal with respect to the second variation $S''_g$ of the total scalar curvature functional, into the four summands
\[\Sy^2(M)=\R g\oplus C^\infty_g(M)g\oplus\im\delta^\ast\oplus\TT(M).\]
These correspond to infinitesimal changes in the metric by homothety, volume-preserving conformal scaling, the action of diffeomorphisms, and moving within $\Sl$, respectively. The second variation $S''_g$ is positive on $C^\infty_g(M)g$, zero on $\im\delta^\ast$ and is given by
\[S''_g(h,h)=-\frac{1}{2}\left(\Delta_Lh-2\Einstein h,h\right)_g\]
on $\TT(M)$, where it has finite coindex and nullity; that is, the maximal subspace of $\TT(M)$ where $S''_g$ is nonnegative is finite-dimensional. In fact, the null directions in $\TT(M)$ are precisely the \emph{infinitesimal Einstein deformations} of $g$, i.e. infinitesimal deformations of $g$ that preserve the Einstein property, the total volume and are orthogonal to the orbit of $g$ under diffeomorphisms.

\begin{defn}
An Einstein metric $g$ on $M$ is called
\begin{enumerate}
	\item \emph{(linearly) stable} (with respect to the Ein\-stein-Hil\-bert action) if $S''_g\leq 0$ on $\TT(M)$ or, equivalently, if $\Delta_L\geq2\Einstein$ on $\TT(M)$. Otherwise it is called \emph{(linearly) unstable}.
	\item \emph{strictly (linearly) stable} (with respect to the Ein\-stein-Hil\-bert action) if $S''_g<0$ on $\TT(M)$ or, equivalently, if $\Delta_L>2\Einstein$ on $\TT(M)$.
	\item \emph{infinitesimally deformable} if $\Delta_Lh=2\Einstein h$ for some nonzero $h\in\TT(M)$.
\end{enumerate}
\end{defn}

\section{Invariant differential operators}
\label{sec:diff}

Let $G$ be a compact Lie group with Lie algebra $\mathfrak{g}$ and $K$ a closed subgroup such that $(M=G/K,g)$ is a reductive Riemannian homogeneous space with $K$-invariant decomposition $\mathfrak{g}=\mathfrak{k}\oplus\mathfrak{m}$, where $\mathfrak{k}$ is the Lie algebra of $K$ and $\mathfrak{m}$ is the reductive complement which is canonically identified with the tangent space $T_{o}M$ at the base point $o:=eK\in M$. Recall that for some representation $\rho: K\to\Aut V$, the \emph{left-regular representation} on the space of $K$-equivariant smooth functions $C^\infty(G,V)^K$ is defined as
\[\ell: G\to\Aut C^\infty(G,V)^K:\ (\ell(x)f)(y):=f(x^{-1}y)\]
for $x,y\in G$. Furthermore, the space $C^\infty(G,V)^K$ is identified with the space of sections of the associated bundle $G\times_\rho V$ over $M$. The identification is given by
\[\Gamma(G\times_\rho V)\to C^\infty(G,V)^K:\ s\mapsto\hat{s},\]
where $\hat{s}$ is defined by $s([x])=[x,\hat{s}(x)]$ for any $x\in G$. If $V$ can be expressed in terms of the isotropy representation $\mathfrak{m}$, then $G\times_\rho V$ is a tensor bundle; for example, we have
\begin{align*}
\X(M)=\Gamma(TM)&\cong\Gamma(G\times_\rho\mathfrak{m})\cong C^\infty(G,\mathfrak{m})^K,\\
\Omega^1(M)=\Gamma(T^\ast M)&\cong\Gamma(G\times_\rho\mathfrak{m}^\ast)\cong C^\infty(G,\mathfrak{m})^K,\\
\Sy^2(M)=\Gamma(\Sym^2T^\ast M)&\cong\Gamma(G\times_\rho\Sym^2\mathfrak{m}^\ast)\cong C^\infty(G,\Sym^2\mathfrak{m})^K,\\
\Sy^2_0(M)=\Gamma(\Sym^2_0T^\ast M)&\cong\Gamma(G\times_\rho\Sym^2_0\mathfrak{m}^\ast)\cong C^\infty(G,\Sym^2_0\mathfrak{m})^K,
\end{align*}
where $\Sym^2_0$, $\Sy^2_0$ denotes the space of trace-free elements with respect to the metric. Note that the invariant Riemannian metric yields an equivalence between $\mathfrak{m}$ and $\mathfrak{m}^\ast$.

Suppose that $V$ is a complex representation. Choose a maximal torus $T$ inside $G$ with Lie algebra $\t$. Recall that up to equivalence, every irreducible finite-dimensional complex representation of $G$ is characterized by its highest weight $\gamma\in\t^\ast$. By the Peter-Weyl theorem and Frobenius reciprocity (cf. \cite{Wa}), the left-regular representation $C^\infty(G,V)^K$ can be decomposed into irreducible summands as\footnote{Here, the bar over the direct sum denotes the closure in $C^\infty(G,V)^K$ (with the $L^2$ inner product). In other words, $\bigoplus_{\gamma}V_\gamma\otimes\Hom_K(V_\gamma,V)$ is dense in $C^\infty(G,V)^K$. In fact, it is dense in $L^2(G,V)^K$, but for our purposes, it suffices to consider smooth sections.}
\begin{equation}
C^\infty(G,V)^K\cong\overline{\bigoplus_{\gamma}}V_\gamma\otimes\Hom_K(V_\gamma,V),\tag{PW}\label{eq:peterweyl}
\end{equation}
where $\gamma$ runs over all highest weights of $G$-representations and $(V_\gamma,\rho_\gamma)$ is the (up to equivalence) unique irreducible representation of $G$ with highest weight $\gamma$. For any
\[\alpha\otimes A\in V_\gamma\otimes\Hom_K(V_\gamma,V),\]
the corresponding element of $C^\infty(G,V)^K$ is defined by
\[f^A_\alpha:\ G\to V:\ x\mapsto A(\rho_\gamma(x^{-1})\alpha).\]

Since the Lichnerowicz Laplacian $\Delta_L$ on $\Gamma(G\times_\rho V)$ is a $G$-invariant differential operator, Schur's Lemma implies that on each of the isotypical subspaces
\[V_\gamma\otimes\Hom_K(V_\gamma,V),\]
$\Delta_L$ acts as an endomorphism of the finite-dimensional vector space $\Hom_K(V_\gamma,V)$, that is,
\[\Delta_Lf^A_\alpha=f^{L_\gamma(A)}_\alpha\]
for some $L_\gamma\in\End\Hom_K(V_\gamma,V)$.

In order to obtain the spectrum of $\Delta_L$, one would have to find the eigenvalues of each $L_\gamma$ -- a potentially very cumbersome task. We will shortly see that this matter is considerably simpler in the symmetric case.

Fix an $\Ad$-invariant inner product $\langle\cdot,\cdot\rangle_{\mathfrak{g}}$ on the Lie algebra $\mathfrak{g}$. If we assume that $G$ is semisimple, one such inner product is given by $-B$, where $B$ is the Killing form on $\mathfrak{g}$, defined by
\[B(X,Y):=\tr(\ad(X)\circ\ad(Y))\]
for $X,Y\in\mathfrak{g}$. Recall that for any representation $\pi: G\to\Aut W$, the Casimir operator $\Cas^G_\pi$ with respect to the chosen inner product is an equivariant endomorphism of $W$, defined as
\[\Cas^G_\pi:=-\sum_id\pi(e_i)\circ d\pi(e_i)\]
for any orthonormal basis $(e_i)$ of $\mathfrak{g}$. 

The following proposition combines two well-known results that allow us to compute the eigenvalues of $\Delta_L$ on compact symmetric spaces, the latter being a formula due to H. Freudenthal (cf. \cite{FH}).

\begin{prop}\label{freudenthal}
Let $(M=G/K,g)$ be a compact Riemannian symmetric space where the Riemannian metric is induced by an $\Ad$-invariant inner product $\langle\cdot,\cdot\rangle_{\mathfrak{g}}$ on $\mathfrak{g}$, and let $\rho: K\to\Aut V$ be a representation.
\begin{enumerate}
	\item On the left-regular representation $\Gamma(G\times_\rho V)$, the Lichnerowicz Laplacian $\Delta_L$ coincides with the Casimir operator $\Cas^G_\ell$ of the representation $\ell: G\to\Aut\Gamma(G\times_\rho V)$.
	\item On each irreducible representation $V_\gamma$, the Casimir eigenvalue is given by
	\[\Cas^G_\gamma=\langle\gamma,\gamma+2\delta_{\mathfrak{g}}\rangle_{\t^\ast},\]
	where $\delta_{\mathfrak{g}}$ is the half-sum of positive roots and $\langle\cdot,\cdot\rangle_{\t^\ast}$ is the inner product on $\t^\ast$ induced by the inner product on $\t\subset\mathfrak{g}$.
\end{enumerate}
\end{prop}

\begin{bem}
The first statement is a consequence of a more general result. Let $G$ be a compact Lie group and $(M=G/K,g)$ be a reductive Riemannian homogeneous space. To the reductive decomposition corresponds a canonical $G$-invariant connection on $M$ (also called the \emph{Ambrose-Singer connection}), which we denote by $\bar\nabla$. This connection in turn defines a curvature tensor $\bar R$ and an analogue to the Lichnerowicz Laplacian via
\[\bar\Delta:=\bar\nabla^\ast\bar\nabla+q(\bar R),\]
called the \emph{standard Laplacian} of this connection (introduced in \cite{SW1}). Then, in fact, $\bar\Delta=\Cas^G_\ell$ on $\Gamma(G\times_\rho V)$. The above statement follows when we note that on Riemannian symmetric spaces, the Ambrose-Singer connection coincides with the Levi-Civita connection.
\end{bem}

According to (\ref{eq:peterweyl}), we can write the complexified left-regular representation on trace-free symmetric $2$-tensors as
\[\Sy^2_0(M)^\C\cong\overline{\bigoplus_{\gamma}}V_\gamma\otimes\Hom_K(V_\gamma,\Sym^2_0\mathfrak{m}^\C).\]
Recall that irreducible symmetric spaces of compact type can be endowed with a Riemannian metric induced by the Killing form (the so-called standard metric). In this case, the critical eigenvalue of $\Delta_L$ is $2\Einstein=1$. Supposing we have a representation $V_\gamma$ with subcritical Casimir eigenvalue $\Cas^G_\gamma<1$ occurring in this decomposition, it remains to check whether the tensors in the corresponding subspace are divergence-free. By Schur's Lemma, the $G$-invariant operator
\[\delta: \Sy^2_0(M)^\C\to\Omega^1(M)^\C\]
is constant on each irreducible subspace. This means that we can regard $\delta$ as a linear mapping
\[\delta: \Hom_K(V_\gamma,\Sym^2_0\mathfrak{m}^\C)\to\Hom_K(V_\gamma,\mathfrak{m}^\C),\]
the so-called \emph{prototypical differential operator} associated to $\delta$ and $V_\gamma$. For a further discussion of invariant differential operators on homogeneous spaces, we refer the reader to Section 2 of \cite{SW}.

The following lemma is of use when we need to calculate $\delta$ explicitly. A derivation of essentially the same formula can also be found in \cite{SW}, Section 2.

\begin{lem}\label{divformula}
Suppose $(M,g)$ is a Riemannian symmetric space. Let $h\in\Sy^2(M)^\C$ correspond to an element
\[\alpha\otimes A\in V_\gamma\otimes\Hom_K(V_\gamma,\Sym^2\mathfrak{m}^\C)\]
in the decomposition {\normalfont (\ref{eq:peterweyl})} of $\Sy^2(M)^\C$. Let further $(e_i)$ be an orthonormal basis of $\mathfrak{m}$. Then we have
\[(\delta h)_{o}(X)=\sum_i\langle A(d\rho_\gamma(e_i)\alpha),e_i\odot X\rangle\]
for any $X\in\mathfrak{m}\cong T_{o}M$.
\end{lem}
\begin{proof}
The element of $C^\infty(G,\Sym^2\mathfrak{m}^\C)^K$ corresponding to $h\in\Sy^2(M)$ is given by
\[\hat{h}=f^A_\alpha:\ G\to\Sym^2\mathfrak{m}^\C:\ x\mapsto A(\rho_\gamma(x^{-1})\alpha),\]
where $\rho_\gamma$ is the representation of $G$ on $V_\gamma$. The covariant derivative of $h$ at the base point may be expressed by
\[(\nabla h)_{o}(X,Y)=\langle d\hat{h}_e,X\odot Y\rangle\]
for $X,Y\in\mathfrak{m}\cong T_{o}M$, since $\nabla$ coincides with the Ambrose-Singer connection on $M$ as a reductive homogeneous space. This implies that
\begin{align*}
(\delta h)_{o}(X)&=-\sum_ie_i\lrcorner\nabla_{e_i}h(X)=-\sum_i\nabla_{e_i}h(e_i,X)=-\sum_i\langle d\hat{h}(e_i),e_i\odot X\rangle\\
&=-\sum_i\langle df_\alpha^A(e_i),e_i\odot X\rangle=\sum_i\langle A(d\rho_\gamma(e_i)\alpha),e_i\odot X\rangle.
\end{align*}
\end{proof}

\section{tt-Eigenspaces of the Lichnerowicz Laplacian}
\label{sec:killing}

We return to the general setting of a compact Einstein manifold $(M,g)$. Define
\[\theta: \Omega^1(M)\to\Sy^2_0(M):\ \alpha\mapsto\delta^\ast\alpha+\frac{2}{n}\delta\alpha\cdot g,\]
so that $\theta\alpha$ is precisely the trace-free part of $\delta^\ast\alpha\in\Sy^2(M)$. The kernel of this operator is (via the metric) isomorphic to the space of \emph{conformal Killing fields} on $(M,g)$, that is, the space of vector fields $X\in\X(M)$ such that $L_Xg=fg$ for some $f\in C^\infty(M)$. We thus call $\theta$ the \emph{conformal Killing operator}.

The following lemma is a generalization of a proposition by Koiso \cite[Prop. 3.3]{Koi2}. For the proof, we refer the reader to the Appendix.

\begin{lem}\label{dimtt}
Let $(M,g)$ be a compact Einstein manifold of dimension $n\geq3$. For any $\lambda\in\R$, the dimension of the eigenspace of $\Delta_L$ to the eigenvalue $\lambda$ on tt-tensors is given by
\begin{align*}
\dim\ker(\Delta_L-\lambda)\big|_{\TT(M)}=&\dim\ker(\Delta_L-\lambda)\big|_{\Sy^2_0(M)}-\dim\ker(\Delta-\lambda)\big|_{\Omega^1(M)}\\
&+\dim\left(\ker(\Delta-\lambda)\big|_{\Omega^1(M)}\cap\ker\theta\right).
\end{align*}
\end{lem}

At first glance, the third term on the right hand side of the above formula does not look very amenable to computation. However, matters are made easier if we observe the following properties of (conformal) Killing vector fields on Einstein manifolds, both of which are proven in the Appendix.

\begin{lem}\label{confkilling}
On any compact Einstein manifold $(M,g)$ not isometric to a standard round sphere, conformal Killing fields are actually Killing, that is, $L_X g=fg$ for some ${f\in C^\infty(M)}$ implies $f=0$. Equivalently, $\ker\theta=\ker\delta^\ast$ on $\Omega^1(M)$.
\end{lem}

\begin{lem}\label{killingcrit}
Any Killing field $X\in\X(M)$ on an Einstein manifold with Einstein constant $\Einstein$ satisfies
\[\Delta X^\flat=2\Einstein X^\flat.\]
Equivalently, $\ker\delta^\ast\subset\ker(\Delta-2\Einstein)$ on $\Omega^1(M)$.
\end{lem}

If we assume that $(M,g)$ is not isometric to a standard sphere, we can immediately conclude that the intersection $\ker(\Delta-\lambda)\big|_{\Omega^1(M)}\cap\ker\theta$ is trivial if $\lambda\neq2\Einstein$. By virtue of Lemma \ref{dimtt}, we obtain the following.

\begin{kor}\label{dimtt2}
Let $(M,g)$ be a compact Einstein manifold that is not isometric to a standard round sphere, and let $\Einstein$ be its Einstein constant. For any $\lambda\neq 2\Einstein$, the dimension of the eigenspace of $\Delta_L$ to the eigenvalue $\lambda$ on tt-tensors is given by
\begin{align*}
\dim\ker(\Delta_L-\lambda)\big|_{\TT(M)}=&\dim\ker(\Delta_L-\lambda)\big|_{\Sy^2_0(M)}-\dim\ker(\Delta-\lambda)\big|_{\Omega^1(M)}.
\end{align*}
\end{kor}

\begin{bem}
If we set $\lambda=2\Einstein$ in Lemma \ref{dimtt} and note that
\[\ker(\Delta-2\Einstein)\big|_{\Omega^1(M)}\cap\ker\theta=\ker\delta^\ast\big|_{\Omega^1(M)}\]
(as Koiso did in his proof of \cite[Prop.~3.3]{Koi2}), we recover the original formula for the critical eigenvalue
\begin{align*}
\dim\ker(\Delta_L-2\Einstein)\big|_{\TT(M)}=\,&\dim\ker(\Delta_L-2\Einstein)\big|_{\Sy^2_0(M)}-\dim\ker(\Delta-2\Einstein)\big|_{\Omega^1(M)}\\
&+\dim\ker\delta^\ast\big|_{\Omega^1(M)}.
\end{align*}
\end{bem}

\begin{bem}
Although the dimension formula of Lemma \ref{dimtt} works on any compact Einstein manifold $(M,g)$, it is worth mentioning that if additionally, $(M,g)$ carries the structure of a Riemannian homogeneous space $M=G/K$, the result can be refined in terms of irreducible representations of $G$. Namely, if $V_\gamma$ is an irreducible representation of $G$, then the multiplicity of $V_\gamma$ in the (complexified) left-regular representation on tt-tensors is given by
\begin{align*}
\dim\Hom_G(V_\gamma,\TT(M)^\C)=\,&\dim\Hom_K(V_\gamma,\Sym^2_0\mathfrak{m}^\C)-\dim\Hom_K(V_\gamma,\mathfrak{m}^\C)\\
&+\dim\Hom_G(V_\gamma,(\ker\theta)^\C).
\end{align*}
As in the proof of Lemma \ref{dimtt}, the dimension formula essentially arises from the short exact sequence
\[0\longrightarrow\ker\theta\stackrel{\subset}{\longrightarrow}\Omega^1(M)\stackrel{\theta}{\longrightarrow}\Sy^2_0(M)\stackrel{P}{\longrightarrow}\TT(M)\longrightarrow0\]
and the fact that the Laplacian commutes with every arrow. In the homogeneous case, we note that we have a short exact sequence of $G$-representations and use Frobenius reciprocity to arrive at the statement.
\end{bem}

\section{The symmetric space $\SU(n)$}
\label{sec:sun}

Throughout what follows, let $n\geq3$. As a symmetric space, $\SU(n)=G/K$ where $G=\SU(n)\times\SU(n)$ and $K=\SU(n)$ is diagonally embedded, i.e. via
\[\SU(n)\hookrightarrow\SU(n)\times\SU(n):\ k\mapsto (k,k).\]
Let $\mathfrak{g}$ and $\mathfrak{k}$ denote the corresponding Lie algebras of $G$ and $K$, respectively. We endow $M$ with the standard metric $g$ induced by the Killing form on $\mathfrak{g}$. Hence, $M$ is Einstein with critical eigenvalue $2\Einstein=1$. The reductive decomposition of $\mathfrak{g}$ with respect to $g$ is given by
\[\mathfrak{g}=\tilde{\mathfrak{k}}\oplus\mathfrak{m},\]
where
\begin{align*}
\tilde{\mathfrak{k}}&=\{(X,X)\,|\,X\in\mathfrak{k}\},\\
\mathfrak{m}&=\{(X,-X)\,|\,X\in\mathfrak{k}\}.
\end{align*}
The $K$-representations $\mathfrak{k}$, $\tilde{\mathfrak{k}}$ and $\mathfrak{m}$ are all equivalent. We denote by $E=\C^n$ the standard representation of $K$.

\begin{lem}\label{uniquesubcrit}
Let $V_\gamma$ be an irreducible complex representation of $G$ with $\Cas_\gamma^G<1$ and
\[\Hom_K(V_\gamma,\Sym^2_0\mathfrak{k}^\C)\neq0.\]
Then $V_\gamma$ is equivalent to one of the $G$-representations $E\otimes E^\ast$ and $E^\ast\otimes E$. In fact,
\[\dim\Hom_K(V_\gamma,\Sym^2_0\mathfrak{k}^\C)=1\]
and the Casimir eigenvalue is $\Cas_\gamma^G=\frac{(n-1)(n+1)}{n^2}$.
\end{lem}
\begin{proof}
Let $\t$ be the torus of diagonal matrices in $\mathfrak{k}$. The dual $\t^\ast$ is generated by the weights $\varepsilon_1,\ldots,\varepsilon_n$ of the defining representation $E$. Explicitly,
\[\varepsilon_j(X)=X_j,\quad 1\leq j\leq n\]
for $X=\diag(\i X_1,\ldots,\i X_n)\in\t$. Note that $\varepsilon_1+\ldots+\varepsilon_n=0$.

Fix the ordering on roots and weights such that the simple roots of $\mathfrak{k}$ are given by
\[\varepsilon_j-\varepsilon_{j+1},\quad1\leq j\leq n-1.\]
The semigroup of dominant integral weights is then generated by the fundamental weights
\[\omega_j=\sum_{k=1}^j\varepsilon_j,\quad 1\leq j\leq n-1,\]
cf. \cite[§15.1]{FH}. The highest weights of representations of $K$, i.e. all the dominant integral weights, are precisely the linear combinations
\[\gamma=\sum_{r=1}^{n-1}a_r\omega_r\]
with coefficients $a_r\in\N_0$. The fundamental weights themselves correspond to the representations
\[V_{\omega_r}=\Lambda^rE\cong\Lambda^{n-r}E^\ast.\]
Let $\gamma,\gamma'\in\t^\ast$ be two dominant integral weights. In particular, they satisfy
\[\langle\gamma,\gamma'\rangle_{\t^\ast}\geq0.\]
Using Freudenthal's formula for the Casimir operator $\Cas^K_\gamma$ of a $K$-re\-pre\-sen\-ta\-tion $V_\gamma$, this implies the estimate
\begin{align*}
\Cas_{\gamma+\gamma'}^K&=\langle\gamma+\gamma'+2\delta_{\mathfrak{k}},\gamma+\gamma'\rangle_{\t^\ast}=\langle\gamma+2\delta_{\mathfrak{k}},\gamma\rangle_{\t^\ast}+2\langle\gamma,\gamma'\rangle_{\t^\ast}+\langle\gamma'+2\delta_{\mathfrak{k}},\gamma'\rangle_{\t^\ast}\\
&\geq\langle\gamma+2\delta_{\mathfrak{k}},\gamma\rangle_{\t^\ast}+\langle\gamma'+2\delta_{\mathfrak{k}},\gamma'\rangle_{\t^\ast}=\Cas_{\gamma}^K+\Cas_{\gamma'}^K.
\end{align*}
In particular, we obtain
\begin{equation}\Cas_\gamma^K\geq\sum_ra_r\Cas_{\omega_r}^K\label{casineq}\tag{$\ast$}\end{equation}
for $\gamma=\sum_{r=1}^{n-1}a_r\omega_r$.

The Casimir eigenvalues of the fundamental representations are given as
\[\Cas_{\omega_r}^K=\frac{(n+1)r(n-r)}{2n^2}\]
for $r=1,\ldots,n-1$. Note that this expression is symmetric around $r=\frac{n}{2}$ and strictly increasing for $r\leq\frac{n}{2}$. Furthermore, we can compute that
\begin{align*}
\Cas^K_{\omega_1}&=\frac{(n+1)(n-1)}{2n^2}<1,\\
\Cas^K_{\omega_2}&=\frac{(n+1)(n-2)}{n^2}<1,\\
\Cas^K_{\omega_3}&=\begin{cases}
\frac{7}{8}<1,&n=6,\\[4pt]
\frac{48}{49}<1,&n=7,\\[4pt]
\frac{3(n+1)(n-3)}{2n^2}>1,&n\geq8,
\end{cases}\\
\Cas^K_{\omega_1}+\Cas^K_{\omega_2}&>1,\ n\geq4,\\
\Cas^K_{2\omega_1}&>1,\\
\Cas^K_{\omega_1+\omega_{n-1}}&=1,
\end{align*}
cf. table on p.~15 of \cite{SW}. Combining the above with inequality (\ref{casineq}), we can deduce that if $\gamma$ is a highest weight with $\Cas_\gamma^K<1$, then necessarily
\[\gamma\in\{0,\omega_1,\omega_{n-1},\omega_2,\omega_{n-2},\underbrace{\omega_3,\omega_{n-3}}_{\text{if }n=6,7}\}.\]
These dominant integral weights are, respectively, highest weights of the representations $\C$, $E$, $E^\ast$, $\Lambda^2E$, $\Lambda^2E^\ast$, $\Lambda^3E$, $\Lambda^3E^\ast$ of $K$.

The irreducible representations of $G=K\times K$ are precisely the tensor products of irreducible representations of $K$. Let $\gamma,\gamma'$ be highest weights of $K$-re\-pre\-sen\-ta\-tions such that
\[\Cas^G_{(\gamma,\gamma')}=\Cas^K_\gamma+\Cas^K_{\gamma'}<1\]
holds. Assuming that $\gamma,\gamma'\neq0$, we conclude that $\gamma,\gamma'\in\{\omega_1,\omega_{n-1}\}$. This yields the four pairwise inequivalent $G$-representations $E\otimes E$, $E\otimes E^\ast$, $E^\ast\otimes E$ and $E^\ast\otimes E^\ast$. Furthermore, in the case of $\gamma=0$ or $\gamma'=0$ we obtain the representations of $K$ that were listed above, composed with the projection onto one factor,
\[G\to K:\ (k_1,k_2)\mapsto k_1\quad\text{or}\quad(k_1,k_2)\mapsto k_2,\]
respectively. By restricting the mentioned $G$-representations to $K$ via the embedding
\[K\to G:\ k\mapsto(k,k),\]
we again obtain the irreducible $K$-representations $\C$, $E$, $E^\ast$, $\Lambda^2E$, $\Lambda^2E^\ast$, $\Lambda^3E$, $\Lambda^3E^\ast$ as well as the tensor product representations $E\otimes E$, $E\otimes E^\ast$ and $E^\ast\otimes E^\ast$. The latter are not irreducible, but decompose into irreducible summands as follows:
\begin{align*}
E\otimes E&=\Sym^2E\oplus\Lambda^2E,\\
E\otimes E^\ast&=E\otimes_0E^\ast\oplus\C,\\
E^\ast\otimes E^\ast&=\Sym^2E^\ast\oplus\Lambda^2E^\ast.
\end{align*}
Here $E\otimes_0E^\ast$ is the set of trace-free elements of $E\otimes E^\ast$ when regarded as $n\times n$-matrices over $\C$. As a representation of $K$, we have
\[E\otimes_0E^\ast\cong V_{\omega_1+\omega_{n-1}}\cong\mathfrak{k}^\C.\]

The $K$-representation $\Sym^2\mathfrak{k}^\C\cong\Sym^2(E\otimes_0E^\ast)$ appears on one hand as a summand of
\[\Sym^2(E\otimes E^\ast)\cong\Sym^2(E\otimes_0E^\ast\oplus\C)\cong\Sym^2(E\otimes_0E^\ast)\oplus E\otimes_0E^\ast\oplus\C.\]
On the other hand, the symmetric power of the tensor product is given by\footnote{This is a consequence of, for example, the formula $\Sym^d(V\otimes W)=\bigoplus\mathbb{S}_\lambda(V)\otimes\mathbb{S}_\lambda(W)$ in \cite[Ex.~6.11]{FH}.}
\[\Sym^2(E\otimes E^\ast)\cong\Sym^2E\otimes\Sym^2E^\ast\oplus\Lambda^2E\otimes\Lambda^2E^\ast.\]
The tensor products $\Sym^2E\otimes\Sym^2E^\ast$ and $\Lambda^2E\otimes\Lambda^2E^\ast$ can in turn be decomposed into
\begin{align*}
\Sym^2E\otimes\Sym^2E^\ast&\cong V_{2\omega_1+2\omega_{n-1}}\oplus V_{\omega_1+\omega_{n-1}}\oplus\C,\\
\Lambda^2E\otimes\Lambda^2E^\ast&\cong\begin{cases}
E^\ast\otimes E\cong V_{\omega_1+\omega_{n-1}}\oplus\C,&n=3,\\
V_{\omega_2+\omega_{n-2}}\oplus V_{\omega_1+\omega_{n-1}}\oplus\C,&n\geq4.
\end{cases}
\end{align*}
By comparing summands, we see that
\[\Sym^2(E\otimes_0E^\ast)\cong V_{2\omega_1+2\omega_{n-1}}\oplus\underbrace{V_{\omega_2+\omega_{n-2}}}_{\text{if }n\geq4}\oplus\, E\otimes_0E^\ast\oplus\C.\]
Hence, the trace-free part is given by
\[\Sym^2_0(E\otimes_0E^\ast)\cong V_{2\omega_1+2\omega_{n-1}}\oplus\underbrace{V_{\omega_2+\omega_{n-2}}}_{\text{if }n\geq4}\oplus\, E\otimes_0E^\ast.\]
Now that we have decomposed the relevant representations into irreducible summands, we recognize that $E\otimes E^\ast$ and $E^\ast\otimes E$ are the only two of the specified subcritical representations of $G$ that, after restriction to $K$, have a common summand with $\Sym^2_0\mathfrak{k}^\C$. In each case, the summand in question $E\otimes_0E^\ast\cong\mathfrak{k}^\C$ appears with multiplicity $1$; hence we have
\[\dim\Hom_K(E\otimes E^\ast,\Sym^2_0\mathfrak{k}^\C)=\dim\Hom_K(E^\ast\otimes E,\Sym^2_0\mathfrak{k}^\C)=1.\]
Moreover, both $G$-representations exhibit the same Casimir eigenvalue
\[\Cas^G_{(\omega_1,\omega_{n-1})}=\Cas^G_{(\omega_{n-1},\omega_1)}=\Cas^K_{\omega_1}+\Cas^K_{\omega_{n-1}}=\frac{(n-1)(n+1)}{n^2}.\]
\end{proof}

According to Lemma \ref{uniquesubcrit}, the only representations of $G$ (up to equivalence) with subcritical Casimir eigenvalue that occur in decomposition (\ref{eq:peterweyl}) of $\Sy^2_0(M)^\C$ are $E\otimes E^\ast$ and $E^\ast\otimes E$, and we have
\[\dim\Hom_K(E\otimes E^\ast,\Sym^2_0\mathfrak{m}^\C)=\dim\Hom_K(E^\ast\otimes E,\Sym^2_0\mathfrak{m}^\C)=1\]
(recall that $\mathfrak{m}\cong\mathfrak{k}$), i.e. the summand occurs with multiplicity $1$. It remains to check whether the tensors in the corresponding subspaces are divergence-free. Since
\[E\otimes E^\ast\cong\mathfrak{k}^\C\oplus\C\]
as a representation of $K$, we have
\[\dim\Hom_K(E\otimes E^\ast,\mathfrak{m}^\C)=\dim\Hom_K(E^\ast\otimes E,\mathfrak{m}^\C)=1,\]
meaning that both summands also occur in the left-regular representation $\Omega^1(M)$ with the same multiplicity. It now follows from Corollary \ref{dimtt2} that
\[\dim\ker(\Delta_L-\lambda)\big|_{\TT(M)}=0\]
for $\lambda=\frac{(n-1)(n+1)}{n^2}$. Since this is the only subcritical eigenvalue on $\Sy^2_0(M)$, we have shown the following.

\begin{prop}\label{sun}
The symmetric space $\SU(n)$ is linearly stable.
\end{prop}

\section{The symmetric space $E_6/F_4$}
\label{sec:ef}

Let $(\H,\circ)$ be the Albert algebra, where $\H$ is the set of Hermitian $3\times3$-matrices over the octonions, i.e.
\[\H:=\left\{\begin{pmatrix}	a&x&\bar{y}\\\bar{x}&b&z\\y&\bar{z}&c\end{pmatrix}\middle|a,b,c\in\R,\ x,y,z\in\O\right\},\]
and with Jordan multiplication defined by
\[X\circ Y:=\frac{1}{2}(XY+YX).\]
The exceptional Lie group $E_6$ can be realized as
\[E_6:=\left\{\alpha\in\Aut_\C\H^\C\,\middle|\,\alpha\text{ preserves determinant and inner product}\right\},\]
while $F_4$ is defined as the set of algebra automorphisms
\[F_4:=\Aut(\H,\circ).\]
By complex-linearly extending linear automorphisms of $\H$, one obtains the inclusion $\Aut_\R\H\subset\Aut_\C\H^\C$. In this sense, we have $F_4\subset E_6$. In fact,
\[F_4=E_6\cap\Aut_\R\H.\]
As a representation of $E_6$, $\H^\C$ is irreducible. As an $F_4$-representation, $\H$ decomposes into the irreducible summands
\[\H\cong\H_0\oplus\R,\]
where $\H_0$ is the set of trace-free elements of $\H$. An invariant inner product on $\H$ is defined by
\[\langle A,B\rangle:=\tr(A\circ B)\]
for $A,B\in\H$. An orthogonal basis of $\H$ (cf. Section 2.1 of \cite{Yo}) is given by the matrices
\[E_1:=\begin{pmatrix}
	1&0&0\\
	0&0&0\\
	0&0&0
\end{pmatrix},\quad E_2:=\begin{pmatrix}
	0&0&0\\
	0&1&0\\
	0&0&0
\end{pmatrix},\quad E_3:=\begin{pmatrix}
	0&0&0\\
	0&0&0\\
	0&0&1
\end{pmatrix},\]
\[F_1(x):=\begin{pmatrix}
	0&0&0\\
	0&0&x\\
	0&\bar{x}&0
\end{pmatrix},\quad F_2(x):=\begin{pmatrix}
	0&0&\bar{x}\\
	0&0&0\\
	x&0&0
\end{pmatrix},\quad F_3(x):=\begin{pmatrix}
	0&x&0\\
	\bar{x}&0&0\\
	0&0&0
\end{pmatrix},\]
where $x$ runs through the standard basis of $\O$ as a real vector space.

In this section, we consider the Riemannian symmetric space $M=E_6/F_4$ equipped with the standard metric (hence with critical eigenvalue $2\Einstein=1$). The reductive decomposition of $\mathfrak{e}_6$ with respect to the standard metric is given by
\[\mathfrak{e}_6=\mathfrak{f}_4\oplus\mathfrak{m},\]
where $\mathfrak{m}\cong\H_0$ as a representation of $F_4$.

\begin{lem}\label{uniquesubcrit2}
Let $V_\gamma$ be an irreducible complex representation of $E_6$ with $\Cas_\gamma^{E_6}<1$ and
\[\Hom_{F_4}(V_\gamma,\Sym^2_0\H_0^\C)\neq0.\]
Then $V_\gamma$ is equivalent to one of the $E_6$-representations $\H^\C$ and $\overline{\H^\C}$. In fact,
\[\dim\Hom_{F_4}(\H^\C,\Sym^2_0\H_0^\C)=\dim\Hom_{F_4}(\overline{\H^\C},\Sym^2_0\H_0^\C)=1,\]
and the Casimir eigenvalue is $\Cas_\gamma^G=\frac{13}{18}$.
\end{lem}
\begin{proof}
We abstain from specifying a particular choice of simple root system and fundamental weights for $E_6$ and $F_4$, since we are merely interested in the corresponding fundamental representations of the respective Lie group. Following the enumerative convention of Bourbaki (as used by the software package LiE), if we denote the fundamental weights of $E_6$ by $\omega_1,\ldots,\omega_6$ and of $F_4$ by $\eta_1,\ldots,\eta_4$, then the associated representations are identified as
\begin{align*}
 V_{\omega_1}&=\mathbf{27}\cong\H^\C,&V_{\omega_2}&=\mathbf{78}\cong\mathfrak{e}_6^\C,&V_{\omega_3}&=\mathbf{351}\cong\Lambda^2\H^\C,\\
 V_{\omega_4}&=\mathbf{2925}\cong\Lambda^3\H^\C,&V_{\omega_5}&=\overline{\mathbf{351}}\cong\Lambda^2\overline{\H^\C},&V_{\omega_6}&=\overline{\mathbf{27}}\cong\overline{\H^\C},\\
 V_{\eta_1}&=\mathbf{52}\cong\mathfrak{f}_4^\C,&V_{\eta_2}&=\mathbf{1274},&V_{\eta_3}&=\mathbf{273},&V_{\eta_4}&=\mathbf{26}\cong\H_0^\C,
\end{align*}
where the number indicates the dimension.

As in the proof of Lemma \ref{uniquesubcrit}, we have the estimate
\[\Cas_\gamma^{E_6}\geq\sum_{r=1}^6a_r\Cas_{\omega_r}^{E_6}\]
for any representation $V_\gamma$ of $E_6$ with highest weight
\[\gamma=\sum_{r=1}^6a_r\omega_r.\]
Among the fundamental representations, only the Casimir eigenvalues
\[\Cas_{\omega_1}^{E_6}=\Cas_{\omega_6}^{E_6}=\frac{13}{18}\]
are smaller than $1$ (see table on p. 16 of \cite{SW}). Since $\frac{13}{18}+\frac{13}{18}>1$, it follows that only the representations to the highest weights $\C,\H^\C,\overline{\H^\C}$ come into question.

Consider now the $F_4$-representation $\H_0^\C\cong V_{\eta_4}$. We obtain\footnote{This has been verified through use of the software package LiE v2.1. See, for example, \url{http://young.sp2mi.univ-poitiers.fr/cgi-bin/form-prep/marc/sym-alt.act?x1=0&x2=0&x3=0&x4=1&power=2&kind=sym&rank=4&group=F4} or enter the command \texttt{sym\_tensor(2,[0,0,0,1],F4)} into the LiE shell.} the decomposition
\[\Sym^2V_{\eta_4}\cong V_{2\eta_4}\oplus V_{\eta_4}\oplus\C\]
into irreducible summands, hence
\[\Sym^2_0\H_0^\C\cong V_{2\eta_4}\oplus\H_0^\C.\]
Furthermore, we have
\[\H^\C\cong\overline{\H^\C}\cong\H_0^\C\oplus\C\]
as a representation of $F_4$. The assertion follows by comparison of summands.
\end{proof}

Lemma \ref{uniquesubcrit2} now tells us that the representations of $E_6$ with subcritical Casimir eigenvalue that occur in decomposition (\ref{eq:peterweyl}) of $\Sy^2_0(M)^\C$ are precisely $\H^\C$ and $\overline{\H^\C}$, both with multiplicity $1$, i.e.
\[\dim\Hom_{F_4}(\H^\C,\Sym^2_0\mathfrak{m}^\C)=\dim\Hom_{F_4}(\overline{\H^\C},\Sym^2_0\mathfrak{m}^\C)=1,\]
since $\mathfrak{m}\cong\H_0$. Again, we have to check whether the tensors in the corresponding subspace are divergence-free. It follows from the decomposition $\H=\H_0\oplus\R$ as a representation of $F_4$ that
\[\dim\Hom_{F_4}(\H^\C,\mathfrak{m}^\C)=\dim\Hom_{F_4}(\overline{\H^\C},\mathfrak{m}^\C)=1,\]
so as in the previous section, the summand has the same multiplicity in the left-regular representation $\Omega^1(M)$. Again, it follows from Corollary \ref{dimtt2} that
\[\dim\ker(\Delta_L-\lambda)\big|_{\TT(M)}=0\]
for $\lambda=\frac{13}{18}$, and since this is the only subcritical eigenvalue on $\Sy^2_0(M)$, we have shown the following, which, together with Prop.~\ref{sun}, finishes the proof of the main theorem.

\begin{prop}\label{e6f4}
The symmetric space $E_6/F_4$ is linearly stable.
\end{prop}

\clearpage
\appendix
\section{Appendix}
\subsection{Proofs of general statements}

\begin{proof}[Proof of Lemma \ref{dimtt}]
The following is a slightly generalized version of the proof of a result by N. Koiso \cite[Prop.~3.3]{Koi2}. We first note that
\[(\delta\alpha\cdot g,h)_g=\int_M\delta\alpha\langle g,h\rangle_g\vol_g=\frac{1}{2}(\delta\alpha,\tr_gh)_g=\frac{1}{2}(\alpha,d\tr_gh)_g\]
for $\alpha\in\Omega^1(M)$, $h\in\Sy^2(M)$, so the formal adjoint of $\theta$ is given by
\[\theta^\ast: \Sy^2_0(M)\to\Omega^1(M):\ h\mapsto\delta h+\frac{1}{n}d\tr_gh.\]

We show that $\theta$ is overdetermined elliptic. The principal symbol of $\theta$ is
\[\sigma_\xi(\theta)\alpha=\sigma_\xi(\delta^\ast)\alpha+\frac{2}{n}\sigma_\xi(\delta)\alpha\cdot g=\xi\odot\alpha-\frac{2}{n}\langle\xi,\alpha\rangle_gg\]
for $\xi,\alpha\in T^\ast_pM$. If $\xi\neq0$, then $\sigma_\xi(\theta)$ is injective: Suppose $\sigma_\xi(\theta)\alpha=0$. Then
\[\xi\odot\alpha=\frac{2}{n}\langle\xi,\alpha\rangle_gg.\]
Take an orthonormal basis $(e_i)$ with respect to $g$ of $T_pM$ and write
\[\xi=\sum_i\xi_ie_i^\flat,\ \alpha=\sum_i\alpha_ie_i^\flat.\]
For $i,j=1,\ldots,n$, it follows that
\[\xi_i\alpha_j+\xi_j\alpha_i=\frac{2}{n}\langle\xi,\alpha\rangle_g\delta_{ij}\]
and so $\xi_i\alpha_j=-\xi_j\alpha_i$ if $i\neq j$, as well as $\xi_i\alpha_i=\xi_j\alpha_j$ for any $i,j$. Then
\[\xi_i^2\alpha_j=-\xi_i\alpha_i\xi_j=-\xi_j^2\alpha_j.\]
If $\alpha_j\neq0$, this would imply that $\xi_i^2+\xi_j^2=0$ and so $\xi_i=\xi_j=0$, which contradicts the assumption that $\xi\neq0$. Overall, we conclude that $\alpha=0$ and thus the injectivity is proven.

From ellipticity, we obtain the orthogonal decomposition
\[\Sy^2_0(M)=\im\theta\oplus\ker\theta^\ast.\]
Let $h\in\ker(\Delta_L-\lambda)\big|_{\Sy^2_0(M)}$. According to the above decomposition, we can write $h$ as
\[h=\theta\alpha+\psi\]
where $\theta^\ast\psi=0$. Then also,
\[\delta\psi=\theta^\ast\psi-\frac{1}{n}d\tr_g\psi=0.\]
Since $(M,g)$ is Einstein, $\Delta_L$ commutes with $\delta$ on $\Sy^2(M)$ and with $\delta^\ast$ on $\Omega^1(M)$ \cite[10.7/10.8]{lichn}. Furthermore $\Delta_L(fg)=(\Delta f)g$ for any $f\in C^\infty(M)$. We conclude that $\Delta_L$ commutes with $\theta$ and $\theta^\ast$ as well. This implies that
\begin{align*}
\theta(\Delta-\lambda)\alpha&=(\Delta_L-\lambda)\theta\alpha=(\Delta_L-\lambda)(h-\psi)=-(\Delta_L-\lambda)\psi,\\
\theta^\ast(\Delta_L-\lambda)\psi&=(\Delta-\lambda)\theta^\ast\psi=0,
\end{align*}
and so
\[\theta^\ast\theta(\Delta-\lambda)\alpha=-\theta^\ast(\Delta_L-\lambda)\psi=0.\]
It follows that
\[\|\theta(\Delta-\lambda)\alpha\|^2_g=\left(\theta^\ast\theta(\Delta-\lambda)\alpha,(\Delta-\lambda)\alpha\right)_g=0\]
and so $\theta(\Delta-\lambda)\alpha=0=(\Delta_L-\lambda)\psi$. In total, $\psi\in\ker(\Delta_L-\lambda)\big|_{\TT(M)}$.

Also, if $h$ is an element of $\ker(\Delta_L-\lambda)\big|_{\TT(M)}$, then
\[\theta^\ast h=\delta h+\frac{1}{n}d\tr_gh=0\]
and so $\psi=h$. This means that the mapping
\[P: \ker(\Delta_L-\lambda)\big|_{\Sy^2_0(M)}\to\ker(\Delta_L-\lambda)\big|_{\TT(M)}:\ h\mapsto\psi\]
defines a projection, and the dimension formula
\[\dim\ker(\Delta_L-\lambda)\big|_{\TT(M)}=\dim\left(\ker(\Delta_L-\lambda)\big|_{\Sy^2_0(M)}\right)-\dim\ker P\]
holds.

By definition, the kernel of $P$ consists of those $h\in\ker(\Delta_L-\lambda)\big|_{\Sy^2_0(M)}$ with $h=\theta\alpha$ for some $\alpha\in\Omega^1(M)$, i.e. $h\in\im\theta$. Hence we know that
\[\ker P=\ker(\Delta_L-\lambda)\big|_{\Sy^2_0(M)}\cap\im\theta.\]
Let $\alpha\in\ker(\Delta-\lambda)\big|_{\Omega^1(M)}$. We have seen that $\Delta_L$ commutes with $\theta$, so it follows that $\theta\alpha\in\ker(\Delta_L-\lambda)\big|_{\Sy^2_0(M)}$ and therefore
\[\theta\left(\ker(\Delta-\lambda)\big|_{\Omega^1(M)}\right)\subset\ker P.\]
Conversely, let $h\in\ker P$. Then there exists some $\alpha\in\Omega^1(M)$ such that $h=\theta\alpha$, and also $h\in\ker(\Delta_L^g-\lambda)\big|_{\Sy^2_0(M)}$. By the ellipticity of the operator $\Delta-\lambda$, we can decompose $\alpha$ into
\[\alpha=\beta+(\Delta-\lambda)\gamma\]
with $\beta\in\ker(\Delta-\lambda)\big|_{\Omega^1(M)}$, $\gamma\in\Omega^1(M)$. Then
\begin{align*}
0&=(\Delta_L-\lambda)\theta\alpha\\
&=(\Delta_L-\lambda)\theta\beta+(\Delta_L-\lambda)\theta(\Delta-\lambda)\gamma\\
&=\theta(\Delta-\lambda)\beta+(\Delta_L-\lambda)^2\theta\gamma\\
&=(\Delta_L-\lambda)^2\theta\gamma.
\end{align*}
Since $\Delta_L$ is self-adjoint, we have
\[\|(\Delta_L-\lambda)\theta\gamma\|^2_g=\left((\Delta_L-\lambda)^2\theta\gamma,\theta\gamma\right)_g=0\]
and thus
\[\theta(\Delta-\lambda)\gamma=(\Delta_L-\lambda)\theta\gamma=0,\]
i.e. $(\Delta-\lambda)\gamma\in\ker\theta$. This implies that $h=\theta\alpha=\theta\beta$, so
\[\theta: \ker(\Delta-\lambda)\big|_{\Omega^1(M)}\to\ker P\]
is surjective and we obtain the dimension formula
\[\dim\ker P=\dim\ker(\Delta-\lambda)\big|_{\Omega^1(M)}-\dim\left(\ker(\Delta-\lambda)\big|_{\Omega^1(M)}\cap\ker\theta\right).\]
\end{proof}

\begin{proof}[Proof of Lemma \ref{confkilling}]
Let $\Einstein$ be the Einstein constant of $(M,g)$. Let $\alpha\in\Omega^1(M)$ such that
\[\theta\alpha=\delta^\ast\alpha+\frac{2}{n}\delta\alpha\cdot g=0.\]
Taking the divergence yields
\[\delta\theta\alpha=\delta\delta^\ast\alpha-\frac{2}{n}d\delta\alpha=0,\]
since $\delta(fg)=-df$ for $f\in C^\infty(M)$. We make use of the well-known Weitzenböck identities
\begin{alignat*}{2}
 \delta\delta^\ast-\delta^\ast\delta&=\nabla^\ast\nabla-q(R)&\quad&\text{on }\Sy^k(M),\\
 \Delta=d^\ast d+dd^\ast&=\nabla^\ast\nabla+q(R)&\quad&\text{on }\Omega^k(M).
\end{alignat*}
For $k=1$ and since $\delta^\ast=d=\nabla$ on functions and $(M,g)$ is Einstein, these amount to
\begin{align*}
 \delta\delta^\ast\alpha-d\delta\alpha=\nabla^\ast\nabla\alpha-\Einstein\alpha,\\
 d^\ast d\alpha+d\delta\alpha=\nabla^\ast\nabla\alpha+\Einstein\alpha.
\end{align*}
Putting these together, we obtain
\begin{align*}
\delta\theta\alpha=\left(1-\frac{2}{n}\right)d\delta\alpha+\nabla^\ast\nabla\alpha-\Einstein\alpha=\left(2-\frac{2}{n}\right)d\delta\alpha+d^\ast d\alpha-2\Einstein\alpha=0.
\end{align*}
Taking the $L^2$ inner product with $\alpha$ then yields
\[\left(2-\frac{2}{n}\right)\|\delta\alpha\|^2_g+\|d\alpha\|^2_g-2\Einstein\|\alpha\|_g=0.\]
If $\Einstein<0$, this directly implies that $\alpha=0$. If $\Einstein=0$, it implies $\delta\alpha=0$ and $d\alpha=0$, and since $\theta\alpha=0$, it follows that $\delta^\ast\alpha=0$. If $\Einstein>0$, then applying the codifferential to $\delta\theta\alpha$ yields
\[\left(2-\frac{2}{n}\right)d^\ast d\delta\alpha+(d^\ast)^2d\alpha-2\Einstein d^\ast\alpha=\left(2-\frac{2}{n}\right)\Delta\delta\alpha-2\Einstein\delta\alpha=0,\]
so $\delta\alpha$ would be an eigenfunction of the Laplacian to the eigenvalue $\frac{\Einstein n}{n-1}=\frac{\scal_g}{n-1}$. By a theorem of Obata \cite[Thm.~D.I.6]{obata}, this eigenvalue can only be attained on the standard sphere, so necessarily $\delta\alpha=0$. It follows again from $\theta\alpha=0$ that $\delta^\ast\alpha=0$.
\end{proof}

\begin{proof}[Proof of Lemma \ref{killingcrit}]
Let $\alpha\in\Omega^1(M)$ such that $\delta^\ast\alpha=0$. Then also $\delta\alpha=0$, since $\delta\alpha=-\tr_g\delta^\ast\alpha=0$. By virtue of the Weitzenböck formulae that were already employed in the proof of Lemma \ref{confkilling}, we conclude that
\[\Delta\alpha=\nabla^\ast\nabla\alpha+\Einstein\alpha=\delta\delta^\ast\alpha-d\delta\alpha+2\Einstein\alpha=2\Einstein\alpha.\]
\end{proof}

\subsection{Alternative proof of the stability of $\SU(n)$}


An alternative method of checking that the prototypical differential operators
\begin{align*}
\delta&: \Hom_K(E\otimes E^\ast,\Sym^2_0\mathfrak{m}^\C)\to\Hom_K(E\otimes E^\ast,\mathfrak{m}^\C),\\
\delta&: \Hom_K(E^\ast\otimes E,\Sym^2_0\mathfrak{m}^\C)\to\Hom_K(E^\ast\otimes E,\mathfrak{m}^\C)
\end{align*}
are injective is an explicit computation by means of Lemma \ref{divformula}. To do so, we first pick out an explicit element
\[A\in\Hom_K(E\otimes E^\ast,\Sym^2_0\mathfrak{m}^\C)\]
and then proceed to compute the divergence on the corresponding subspace of $\Sy^2_0(M)$.

\begin{lem}\label{sunpi}
Let $\pi: \Sym^2(E\otimes E^\ast)\to E\otimes E^\ast$ denote the mapping defined by
\[\pi(A\odot B):=AB^\ast+BA^\ast,\]
where $A,B\in E\otimes E^\ast$ are regarded as complex $n\times n$-matrices. Then
\[\pi\in\Hom_K(\Sym^2(E\otimes E^\ast),E\otimes E^\ast).\]
Moreover, the restriction
\[\pi: \Sym^2_0(E\otimes_0E^\ast)\to E\otimes_0E^\ast\]
is surjective, and $W:=\left(\ker\pi\big|_{\Sym^2_0(E\otimes_0E^\ast)}\right)^\bot\cong E\otimes_0E^\ast$.
\end{lem}
\begin{proof}
The equivariance of $\pi$ under the action of $K$ follows from
\[\pi(kAk^{-1}\odot kBk^{-1})=kAk^{-1}(k^{-1})^\ast B^\ast k^\ast+k^{-1}Bk(k^{-1})^\ast A^\ast k^\ast=k(AB^\ast+BA^\ast)k^{-1}\]
for any $k\in K=\SU(n)$ and $A,B\in E\otimes E^\ast$. Furthermore, we have
\[\tr(\pi(A\odot B))=\tr(AB^\ast+BA^\ast)=\langle A,B\rangle+\langle B,A\rangle=\tr(A\odot B),\]
where the last trace is taken with respect to the inner product on $E\otimes E^\ast$. This means that
\[\pi(\Sym^2_0(E\otimes E^\ast))\subset E\otimes_0E^\ast.\]
Next we want to show that $\pi$ does not vanish when restricted to $\Sym^2_0(E\otimes_0E^\ast)$. If we denote by $E_{ij}$ the $n\times n$-matrix that has entry $1$ at position $(i,j)$ and $0$ elsewhere, then we have for example $E_{21},E_{31}\in E\otimes_0E^\ast$ and $\langle E_{21},E_{31}\rangle=0$, so $E_{21}\odot E_{31}\in\Sym^2_0(E\otimes_0E^\ast)$ and
\[\pi(E_{21}\odot E_{31})=E_{21}E_{13}+E_{31}E_{12}=E_{23}+E_{32}\neq0.\]
Now, since $E\otimes_0E^\ast$ is irreducible, the mapping
\[\pi: \Sym^2_0(E\otimes_0E^\ast)\to E\otimes_0E^\ast\]
must be surjective. We have seen in the proof of Lemma \ref{uniquesubcrit} that $E\otimes_0E^\ast$ appears in the decomposition of $\Sym^2_0(E\otimes_0E^\ast)$ with multiplicity $1$; hence $W:=\left(\ker\pi\big|_{\Sym^2_0(E\otimes_0E^\ast)}\right)^\bot$ must be the irreducible summand of $\Sym^2_0(E\otimes_0E^\ast)$ that is equivalent to $E\otimes_0E^\ast$.
\end{proof}

\begin{proof}[Alternative proof of Prop. \ref{sun}]
The properties of $\pi$ from Lemma \ref{sunpi} allow us to define
\[\tilde{A}:=\pi\big|_W^{-1}\in\Hom_K(E\otimes_0E^\ast,\Sym^2_0(E\otimes_0E^\ast))\]
and extend it with zero to a mapping $\tilde{A}\in\Hom_K(E\otimes E^\ast,\Sym^2_0(E\otimes_0E^\ast))$. Via the identification $\mathfrak{m}^\C\cong E\otimes_0E^\ast$, this gives rise to a mapping
\[A\in\Hom_K(E\otimes E^\ast,\Sym^2_0\mathfrak{m}^\C).\]
From the equivariance of $\pi\big|_W$, the irreducibility of $W\cong E\otimes_0E^\ast$ and Schur's Lemma it follows that $\pi\big|_W$ is unitary up to a positive constant, that is
\[\langle\pi(v),\pi(w)\rangle_{E\otimes_0E^\ast}=c\cdot\langle v,w\rangle_{\Sym^2_0(E\otimes_0E^\ast)}\]
for all $v,w\in W$ and some $c>0$. Denote the tensor product representation of $G$ on $E\otimes E^\ast$ by
\[\rho: G\to\Aut(E\otimes E^\ast):\ \rho(k_1,k_2)F=k_1Fk_2^{-1}\]
for $F\in E\otimes E^\ast$. Its differential is given by
\[d\rho: \mathfrak{g}\to\End(E\otimes E^\ast):\ d\rho(X_1,X_2)F=X_1F-FX_2\]
for $X_1,X_2\in\mathfrak{k}$. In particular,
\[d\rho(X,-X)F=XF+FX.\]
Let $(e_i)$ be an orthonormal basis of $\mathfrak{m}$, $e_i=(f_i,-f_i)$ with $f_i\in\mathfrak{k}$. Under the identification $\mathfrak{m}^\C\cong E\otimes_0E^\ast$, the invariant inner product changes by some positive constant factor, and $e_i$ is mapped to $f_i$. Hence, $(f_i)$ is an orthonormal basis of $\mathfrak{k}\subset E\otimes_0E^\ast$ up to a positive factor.

Now, let $X\in\mathfrak{k}$ and $F\in E\otimes E^\ast$. Using the formula from Lemma \ref{divformula}, it follows that
\begin{align*}
(\delta h)_{o}(X,-X)&=\sum_i\langle A(d\rho(e_i)F),e_i\odot(X,-X)\rangle_{\Sym^2_0\mathfrak{m}^\C}\\
&=c\cdot\sum_i\langle\tilde{A}(f_iF+Ff_i),f_i\odot X\rangle_{\Sym^2_0(E\otimes_0E^\ast)}\\
&=c\cdot\sum_i\langle\tilde{A}(f_iF+Ff_i),\pr_W(f_i\odot X)\rangle_{\Sym^2_0(E\otimes_0E^\ast)}\\
&=c'\cdot\sum_i\langle f_iF+Ff_i,\pi(\pr_{\Sym^2_0(E\otimes_0E^\ast)}(f_i\odot X))\rangle_{E\otimes_0E^\ast}
\end{align*}
for some $c,c'>0$. Since the trivial summand of $\Sym^2(E\otimes_0E^\ast)$ can only be mapped to the trivial summand of $E\otimes E^\ast$ under the equivariant map $\pi$, we have
\[\pi\circ\pr_{\Sym^2_0(E\otimes_0E^\ast)}=\pr_{E\otimes_0E^\ast}\circ\,\pi\]
on $\Sym^2(E\otimes_0E^\ast)$, implying that
\begin{align*}
(\delta h)_{o}(X,-X)&=c'\cdot\sum_i\langle f_iF+Ff_i,\pr_{E\otimes_0E^\ast}(f_iX^\ast+Xf_i^\ast)\rangle\\
&=-c'\cdot\sum_i\langle f_iF+Ff_i,\pr_{E\otimes_0E^\ast}(f_iX+Xf_i)\rangle.
\end{align*}
Choose the (up to a positive factor) orthonormal basis $(f_i)$ of $\mathfrak{k}$ in such a way that $f_1=E_{12}-E_{21}$. Furthermore, let $X=F=E_{13}-E_{31}$. Then,
\[f_1F+Ff_1=(E_{12}-E_{21})(E_{13}-E_{31})+(E_{13}-E_{31})(E_{12}-E_{21})=-E_{23}-E_{32}\in E\otimes_0E^\ast\]
and we obtain
\begin{align*}
&\sum_i\langle f_iF+Ff_i,\pr_{E\otimes_0E^\ast}(f_iX+Xf_i)\rangle=\sum_i\langle f_iF+Ff_i,\pr_{E\otimes_0E^\ast}(f_iF+Ff_i)\rangle\\
\geq&\langle f_1F+Ff_1,\pr_{E\otimes_0E^\ast}(f_1F+Ff_1)\rangle=\langle E_{23}+E_{32},E_{23}+E_{32}\rangle=2>0.
\end{align*}
In particular, we have found $Y\in\mathfrak{m}$ such that $(\delta h)_{o}(Y)\neq0$, where $h\in\Sy^2_0(M)$ is associated to
\[F\otimes A\in(E\otimes E^\ast)\otimes\Hom_K(E\otimes E^\ast,\Sym^2_0\mathfrak{m}^\C).\]
This means that the linear mapping
\[\delta: \Hom_K(E\otimes E^\ast,\Sym^2_0\mathfrak{m}^\C)\to\Hom_K(E\otimes E^\ast,\mathfrak{m}^\C)\]
is nonzero. Hence, there are no tt-eigentensors for the subcritical Casimir eigenvalue. This proves the assertion.
\end{proof}

\subsection{Alternative proof of the stability of $E_6/F_4$}

As we did before in the situation of $\SU(n)$, we want to apply Lemma \ref{divformula} to verify that the mappings
\begin{align*}
\delta&: \Hom_{F_4}(\H^\C,\Sym^2_0\mathfrak{m}^\C)\to\Hom_{F_4}(\H^\C,\mathfrak{m}^\C),\\
\delta&: \Hom_{F_4}(\overline{\H^\C},\Sym^2_0\mathfrak{m}^\C)\to\Hom_{F_4}(\overline{\H^\C},\mathfrak{m}^\C)
\end{align*}
are injective. Surprisingly, the computation works very similar to the $\SU(n)$ case.

\begin{lem}\label{efpi}
Let $\pi: \Sym^2\H\to\H$ denote the mapping defined by
\[\pi(A\odot B):=AB+BA=2A\circ B.\]
Then we have
\[\pi\in\Hom_{F_4}(\Sym^2\H_0,\H).\]
The restriction
\[\pi: \Sym^2_0\H_0\to\H_0\]
is surjective, and $W:=\left(\ker\pi\big|_{\Sym^2_0\H_0}\right)^\bot\cong\H_0$.
\end{lem}
\begin{proof}
The proof is completely analogous to the proof of Lemma \ref{sunpi}. First, we note that $\pi$ is well-defined since $(\H,\circ)$ is a commutative algebra. The equivariance of $\pi$ under the action of $F_4$ follows from
\[\pi(f(A)\odot f(B))=2f(A)\circ f(B)=f(2A\circ B)=f(\pi(A\odot B))\]
for any $f\in F_4=\Aut(\H,\circ)$ and $A,B\in\H$. Furthermore, we have
\[\tr(\pi(A\odot B))=2\tr(A\circ B)=2\langle A,B\rangle=\tr(A\odot B),\]
where the last trace is taken with respect to the inner product on $\H$. This means that
\[\pi(\Sym^2_0\H)\subset\H_0.\]
Now we want to show that $\pi$ does not vanish when restricted to $\Sym^2_0\H_0$. For example, take $F_1(1),F_2(1)\in\H_0$. We have $\langle F_1(1),F_2(1)\rangle=0$ and thus $F_1(1)\odot F_2(1)\in\Sym^2_0\H_0$. Also,
\[\pi(F_1(1)\odot F_2(1))=2F_1(1)\circ F_2(1)=F_3(1)\neq0.\]
Since $\H_0$ is irreducible over $F_4$, the mapping
\[\pi: \Sym^2_0\H_0\to\H_0\]
must be surjective. From the proof of Lemma \ref{uniquesubcrit2}, we know that $\H_0$ appears in the decomposition of $\Sym^2_0\H_0$ with multiplicity $1$; hence $W:=\left(\ker\pi\big|_{\Sym^2_0\H_0}\right)^\bot$ must be the irreducible summand of $\Sym^2_0\H_0$ that is equivalent to $\H_0$.
\end{proof}

\begin{proof}[Alternative proof of Prop. \ref{e6f4}]
By Lemma \ref{efpi}, we can define
\[A:=\pi\big|_W^{-1}\in\Hom_{F_4}(\H_0,\Sym^2_0\H_0),\]
extend it with zero to $\H$ and then complex-linearly to a mapping $A\in\Hom_{F_4}(\H^\C,\Sym^2_0\H_0^\C)$. Again, we need that $\pi\big|_W$ is unitary up to a positive constant, which follows by Schur's Lemma from the equivariance of $\pi\big|_W$ and the irreducibility of $W\cong\H_0$. By Theorem 3.2.4 in \cite{Yo}, every element  $\alpha\in\mathfrak{e}_6\subset\End_\C(\H^\C)$ can be written as
\[\alpha=\beta+\i T\circ\]
with unique elements $\beta\in\mathfrak{f}_4\subset\mathfrak{e}_6$ and $T\in\H_0$. This corresponds to the $F_4$-invariant decomposition
\[\mathfrak{e}_6\cong\mathfrak{f}_4\oplus\H_0.\]
Throughout what follows, we identify $\mathfrak{m}\cong\H_0$. If we denote the defining representation by
\[\rho: E_6\to\Aut\H^\C,\]
then in particular,
\[d\rho(X)=\i X\circ\]
for $X\in\mathfrak{m}$. Let $(e_i)$ be an orthonormal basis of $\H_0$ (again, under the identification $\mathfrak{m}\cong\H_0$, the invariant inner product changes at most by some positive constant factor), $X\in\mathfrak{m}$ and $F\in\H^\C$. Using Lemma \ref{divformula}, we thus obtain
\begin{align*}
&(\delta h)_{o}(X)=c\cdot\sum_i\langle A(d\rho(e_i)F),e_i\odot X\rangle_{\Sym^2_0\H_0^\C}=c\cdot\sum_i\langle A(\i e_i\circ F),e_i\odot X\rangle_{\Sym^2_0\H_0^\C}\\
=\,&c\cdot\sum_i\langle A(\i e_i\circ F),\pr_W(e_i\odot X)\rangle_{\Sym^2_0\H_0^\C}=c'\cdot\sum_i\langle\i e_i\circ F,\pi(\pr_{\Sym^2_0\H_0}(e_i\odot X))\rangle_{\H_0^\C}
\end{align*}
for some $c,c'>0$. The trivial summand of $\Sym^2\H_0$ can only be mapped to the trivial summand of $\H$ under the equivariant map $\pi$, implying that
\[\pi\circ\pr_{\Sym^2_0\H_0}=\pr_{\H_0}\circ\pi\]
on $\Sym^2\H_0$. Thus, we have
\[(\delta h)_{o}(X)=\i c'\cdot\sum_i\langle e_i\circ F,\pr_{\H_0}(\pi(e_i\odot X))\rangle=2\i c'\sum_i\langle e_i\circ F,\pr_{\H_0}(e_i\circ X)\rangle.\]

Now let $X=F=F_1(1)$. Choose the (up to a positive factor) orthonormal basis $(e_i)$ of $\H_0$ in such a way that $e_1=F_2(1)$. Then we have
\[e_1\circ F=F_2(1)\circ F_1(1)=\frac{1}{2}F_3(1)\in\H_0\]
and it follows that
\begin{align*}
\sum_i\langle e_i\circ F,\pr_{\H_0}(e_i\circ X)\rangle&=\sum_i\langle e_i\circ F,\pr_{\H_0}(e_i\circ F)\rangle\geq\langle e_1\circ F,\pr_{\H_0}(e_1\circ F)\rangle\\
&=\frac{1}{4}\left\langle F_3(1),F_3(1)\right\rangle=\frac{1}{2}>0.
\end{align*}
In particular, we have found $Y\in\mathfrak{m}$ such that $(\delta h)_{o}(Y)\neq0$, where $h\in\Sy^2_0(M)$ is associated to
\[F\otimes A\in\H^\C\otimes\Hom_{F_4}(\H^\C,\Sym^2_0\mathfrak{m}^\C).\]
This means that the linear mapping
\[\delta: \Hom_{F_4}(\H^\C,\Sym^2_0\mathfrak{m}^\C)\to\Hom_{F_4}(\H^\C,\mathfrak{m}^\C)\]
is nonzero. The same argument works for the $E_6$-representation $\overline{\H^\C}$, since we exclusively used real elements and automorphisms in the computation. In total, there are no tt-eigentensors for the subcritical Casimir eigenvalue, which proves the assertion.
\end{proof}

\end{document}